\documentclass[12pt,oneside,openany,article]{memoir}
\usepackage{mempatch}

\nouppercaseheads 
\usepackage[amsthm,thmmarks,hyperref]{ntheorem}
\usepackage{xr-hyper}

\usepackage[noTeX]{mmap}
\usepackage{etex}

\usepackage[english]{babel}
\usepackage[leqno]{mathtools}

\usepackage{caption}
\usepackage{mathrsfs}
\usepackage{upgreek}
\usepackage[compress,square,comma,numbers]{natbib}
\usepackage{hypernat} 

\usepackage{sfmath}

\usepackage{microtype}

\usepackage{subfig}
\usepackage{wrapfig}
\usepackage{array}

\usepackage{graphicx,color}
\usepackage{eso-pic}
\pagecolor{white}
\usepackage{pdfpages}

\usepackage{hyperxmp}
\usepackage{zref}
\usepackage{nameref}
\usepackage[pdftex,bookmarks,pdfnewwindow,plainpages=false,unicode]{hyperref}

\usepackage{bookmark}
\usepackage{attachfile}
\usepackage{embedfile}
\usepackage[3D]{movie15}
\usepackage{animate}

\usepackage[verbose]{placeins}

\usepackage{enumitem}
\usepackage{tikz}
\usepackage{relsize}

\usepackage{amssymb} 

\hypersetup{
colorlinks=true,
linkcolor=black,
citecolor=black,
urlcolor=blue,
pdfauthor={Victor Ivrii},
pdftitle={Local trace asymptotics  in the self-generated magnetic field},
pdfsubject={Sharp Spectral Asymptotics},
pdfkeywords={Schrodinger operator, self-generated magnetic field},
bookmarksdepth={4}
}

\numberwithin{equation}{chapter}

\theoremstyle{plain}
\newtheorem{theorem}{Theorem}[chapter]

\newtheorem{proposition}[theorem]{Proposition}
\newtheorem{corollary}[theorem]{Corollary}

\theoremstyle{definition}

\theoremstyle{remark}
\newtheorem{remark}[theorem]{Remark}

\DeclareMathAlphabet{\mathpzc}{OT1}{pzc}{m}{it}

 \newcommand{\sC}{\mathscr{C}}

 \newcommand{\sH}{\mathscr{H}}

 \newcommand{\sL}{\mathscr{L}}

\newcommand{\Weyl}{{\mathsf{Weyl}}}

\newcommand{\E}{{\mathsf{E}}}

\newcommand{\T}{{\mathsf{T}}}

\newcommand{\const}{{\mathsf{const}}}
\newcommand{\dist}{{{\mathsf{dist}}}}

\newcommand{\loc}{{{\mathsf{loc}}}}
\newcommand{\new}{{{\mathsf{new}}}}

\newcommand{\bR}{{\mathbb{R}}}

\def\1{\boldsymbol {|}}
\newcommand{\boldsigma}{{\boldsymbol{\sigma}}}

\newcommand{\Def}{\mathrel{\mathop:}=}

\renewcommand{\Im}{\operatorname{Im}}       

\newcommand{\supp}{\operatorname{supp}}

\newcommand{\tr}{\operatorname{tr}}
\newcommand{\Tr}{\operatorname{Tr}}

\newcommand{\Res}{\operatorname{Res}}

\newenvironment{claim}[1][{\textup{(\theequation)}}]{\refstepcounter{equation}\vglue10pt
\begin{trivlist}
\item[{\hskip\labelsep#1}]}{\vglue10pt\end{trivlist}}

\newenvironment{claim*}[1][{}]{\vglue10pt
\begin{trivlist}
\item[{\hskip\labelsep#1}]}{\vglue10pt\end{trivlist}}

\newcounter{note}

\DeclareTextCommand{\textinfty}{PU}{\9042\036}

\DeclareTextCommand{\textge}{PU}{\9042\145}
\DeclareTextCommand{\textle}{PU}{\9042\144}
\DeclareTextCommand{\texthat}{PD1}{\136}

\begin{document}


\title{Local trace asymptotics  in the self-generated magnetic field}
\author{Victor Ivrii\thanks{Department of Mathematics, University of Toronto, 40 St. George Street, Toronto, ON, M5S 2E4, Canada, ivrii@math.toronto.edu}}
\maketitle

\chapter{Global theory}
\label{sect-1}

\section{Statement of the problem}%
\label{sect-1-1}
Let us consider the following operator (quantum Hamiltonian) in 
$\bR^d$ with $d=3$
\begin{equation}
H=H _{A,V}=\bigl((hD -A)\cdot \boldsigma \bigr) ^2-V(x)
\label{1-1}
\end{equation}
where $A,V$ are real-valued functions and $V\in \sL^{\frac{5}{2}}\cap \sL^{4}$, 
$A\in \sH^1$. Than this operator is self-adjoint. We are interested in $\Tr^- H_{A,V}$ (the sum of all negative eigenvalues of this operator). Let
\begin{gather}
\E ^*=\inf_{A\in \sH^1_0(B(0,1))}\E(A),\label{1-2}\\
\E(A)\Def \Bigl( \Tr^-H_{A,V}  +  
\kappa^{-1} h^{-2}\int  |\partial A|^2\,dx\Bigr)
\label{1-3}
\end{gather}
with $\partial A=(\partial_i A_j)$ a matrix. Semiclassical asymptotics for such objects were studied first in \cite{EFS2} as a first step to consider asymptotics of the ground state energy for atoms  and molecules   in the self-generated magnetic fields \cite{ES,EFS1} which was achieved in \cite{EFS3} where Scott correction term was recovered. 

Similarly, this paper is the first step to the recovering sharper asymptotics of the ground state energy for atoms  and molecules   in the self-generated magnetic fields, hopefully up to and including Dirac and Schwinger corrections. To do this we improve theorem \ref{EFS2-thm:secondorder} of \cite{EFS2} (see theorems~\ref{thm-2-5}, \ref{thm-2-7}, \ref{thm-4-1} and \ref{thm-4-2} below).

The estimate from above is delivered by $A=0$ and Weyl formula with an error $O(h^{-1})$ as $V\in \sC^{2,1}$\,\footnote{\label{foot-1} Which means that the second derivatives of $V$ are continuous with the continuity modulus $O(|\log |x-y||^{-1})$, see section~\ref{book_new-sect-4-5} of \cite{futurebook}). If there is a boundary it does not pose any problem as it is in the classically forbidden region.}
\begin{equation}
\E^*\le \Weyl_1 + O(h^{-1});
\label{1-4}
\end{equation}
where 
\begin{gather}
\Weyl(\tau) = \frac{1}{3\pi^2} h^{-3}\int (V+\tau)_+^{\frac{3}{2}}\,dx,\label{1-5}\\
\Weyl_1 = \int _{-\infty}^0 \tau\,d_\tau \Weyl(\tau)=
\frac{2}{15\pi^2} \int V_+^{\frac{5}{2}}\,dx.\label{1-6}
\end{gather}
Also for estimates $o(h^{-2})$ we need to include into $\Weyl_1$ the corresponding boundary term. The purpose of this paper is to provide an estimate from below
\begin{equation}
\E^*\ge \Weyl_1 - O(h^{-1});
\label{1-7}
\end{equation}
We will use also $\Weyl (x,\tau)$ and $\Weyl _1(x)$ defined the same way albeit without integration with respect to $x$.

In this version (v3) we fix several significant errors, include asymptotics with $o(h^{-1})$ remainder and provide a bit more details.

\section{Preliminary}
\label{sect-1-2}
Let us estimate from below. First we need the following really simple theorem (cf.  \ref{EFS1-thm:stab} \cite{EFS1})

\begin{theorem}\label{thm-1-1}
Let  $V\in \sL^{\frac{5}{2}}\cap \sL^4$. Then
\begin{gather} 
\E^*\ge -C h^{-3}\label{1-8}\\
\shortintertext{and either}
\frac {1}{\kappa h^2} \int |\partial A|^2\,dx \le Ch^{-3}\label{1-9}.
\end{gather}
or $\E (A) \ge ch^{-3}$.
\end{theorem}

\begin{proof}
Using the Magnetic Lieb-Thirring inequality (\ref{LLS-eq:lt2})  of \cite{LLS})
\begin{multline}
\int \tr e_1(x,x,\tau)\, dx \ge \\
- Ch^{-3} \int V_+^{\frac{5}{2}}\,dx 
-Ch^2 \int \Bigl(h^{-2}\int |\partial A|^2\,dx\Bigr)^{\frac{3}{4}}
\Bigl(h^{-8} \int V_+^4\,dx \Bigr)^{\frac{1}{4}}
\label{1-10}
\end{multline}
we conclude that for any $\delta>0$ 
\begin{equation}
\E(A) \ge -Ch^{-3}-C\delta^3 h^{-3} +
 \bigl(\kappa^{-1}-\delta^{-1}) h^{-1}\int |\partial A|^2\,dx
\label{1-11}
\end{equation}
which implies both statements of the theorem.
\end{proof}

\begin{theorem}\label{thm-1-2}
Let  $V_+\in  \sL^{\frac{5}{2}}\cap \sL^4$, $\kappa \le ch^{-1} $ and 
\begin{equation}
V\le -K^{-1} (1+|x|)^\delta +K.
\label{1-12}
\end{equation}
Then there exists a minimizer $A$.
\end{theorem}

\begin{proof}
Consider a minimizing sequence $A_j$. Without any loss of the generality one can assume that $A_j\to A_\infty$ weakly in $\sH^1$ and  in $\sL^6$ and strongly in $\sL^p_\loc$ with any $p<6$. Then $A_\infty$ is a minimizer.

Really, due to (\ref{1-10}) non-positive spectra of $H_{A_j,V}$ are discrete and the number of  non-positive eigenvalues is bounded by $N_{h}$. Without any loss of the generality one can assume that $\lambda_{j,k}$ have limits (we go to the subsequence if needed) which are either $<0$ or $=0$. Here $\lambda_{j,k}$ are ordered eigenvalues of $H_{A_j,V}$. 

We claim that those limits are also eigenvalues and if $\lambda_{j,k},\ldots,\lambda_{j, k+r-1}$ have the same limit 
$\bar{\lambda}\le 0$, it is eigenvalue of at least multiplicity $r$. Indeed, let $u_{j,k}$ be corresponding eigenfunctions, orthonormal in $\sL^2$. Then in virtue of $A_j$ being bounded in $\sL^6$ and $V\in \sL^4$ we can estimate 
\begin{equation*}
\|Du_{j,k}\|\le K \|u_{j,k}\|_{6}^{1-\sigma}\cdot \|u_{j,k}\|^\sigma\le K\|Du_{j,k}\|^{1-\sigma}\cdot \|u_{j,k}\|^\sigma 
\end{equation*}
with $\sigma>0$ which implies $\|Du_{j,k}\|\le K$. Also assumption (\ref{1-12}) implies that $\| (1+|x|)^{\sigma/2 }u_{j,k}\|$ are bounded and therefore without any loss of the generality one can assume that $u_{j,k}$ converge strongly.

Then 
\begin{equation*}
\lim_{j\to\infty} \Tr^- H_{A_j,V} \ge \Tr^- H_{A_\infty,V},\qquad
\liminf\int |\partial A_j|^2\, dx \ge 
\int |\partial A_\infty |^2\, dx
\end{equation*}
and therefore $\E(A_\infty)\le \E^*$ (and then it is a minimizer and there are equalities and in particular there no other eigenvalues of $H_{A_\infty, V}$. 
\end{proof}

\begin{remark}\label{rem-1-3}
We don't know if the minimizer is unique. Also we do not impose here any restrictions on $K$ (which may depend on $h$) in (\ref{1-12}) or $\kappa>0$. From now on until further notice let $A$ be a minimizer.
\end{remark}

\begin{proposition}\label{prop-1-4}
Let $A$ be a minimizer. Then
\begin{multline}
\frac{2}{\kappa h^2} \Delta A_j (x)   = \Phi_j\Def\\
-\sum_k \bigl(\upsigma_j\upsigma_k (hD_k-A_k)_x  + 
\upsigma_k\upsigma_j (hD_k-A_k)_y  \bigr) e (x,y,\tau) |_{y=x}
\label{1-13}
\end{multline}
where $A=(A_1,A_2,A_3)$, $\boldsigma=(\upsigma_1, \upsigma_2, \upsigma_3)$ and $e(x,y,\tau)$ is the Schwartz kernel of the spectral projector of $H_{A,V}$.
\end{proposition}

\begin{proof}
Consider variation $\updelta A$ of $A$ and variation of $\Tr^- (H)$. Note that the spectral projector of $H$ is 
\begin{equation}
\uptheta (\tau- H) = 
\frac{1}{2\pi i} \int_{-\infty}^\tau \Res_\bR (\tau -H)^{-1}
\label{1-14}
\end{equation}
and therefore
\begin{multline*}
\updelta \Tr \uptheta (\tau- H) = \frac{1}{2\pi i} \int_{-\infty}^\tau \Res_\bR \Tr (\tau -H)^{-1} (\updelta H) (\tau -H)^{-1}= \\
\frac{1}{2\pi i} \int_{-\infty}^\tau \Res_\bR \Tr  (\updelta H) (\tau -H)^{-2} = 
-\partial_\tau \frac{1}{2\pi i} \int_{-\infty}^\tau \Res_\bR \Tr  (\updelta H) (\tau -H)^{-1} = \\
-\partial_\tau \Tr (\updelta H) \uptheta (\tau-H).
\end{multline*}
Plugging it into
\begin{equation}
\Tr^- (H)= \int _{-\infty}^0 \tau d_\tau \Tr \uptheta (\tau- H)=
-\int _{-\infty}^0 \Tr \uptheta (\tau- H)\,d\tau
\label{1-15}
\end{equation}
and integrating  with respect to $\tau$ we arrive after simple calculations to
\begin{equation}
\updelta \Tr^- (H)=\Tr (\updelta H) \uptheta (\tau- H)=-\int \Phi (x) \updelta A(x)\,dx
\label{1-16}
\end{equation}
where $\Phi(x)$ is the right-hand expression of (\ref{1-13}). Therefore
\begin{equation}
\updelta \E(A)= \int \bigl(-\Phi (x)-\frac{2}{\kappa h^2} \Delta A(x)\bigr) \updelta A(x)\,dx
\label{1-17}
\end{equation}
which implies (\ref{1-13}).
\end{proof}

\begin{proposition}\label{prop-1-5}
If for  $\kappa=\kappa^*$ 
\begin{gather}
\E^* \ge \Weyl_1 - CM\label{1-18}\\
\intertext{with $M\ge C h^{-1}$ then for $\kappa \le \kappa^*(1-\epsilon_0)$}
\frac{1}{\kappa h^2} \int |\partial A|^2\,dx \le C_1M.\label{1-19}
\end{gather}
\end{proposition}
\begin{proof}
Proof is obvious based also on the upper estimate $\E^*\le \Weyl_1+Ch^{-1}$.
\end{proof}

\section{Estimates. I}
\label{sect-1-3} 

\begin{proposition}\label{prop-1-6}
Let \textup{(\ref{1-19})} be fulfilled and let
\begin{equation}
\varsigma = \kappa M h^{\frac{3}{2}} \le c
\label{1-20}
\end{equation}
Then as $\tau\le c$ 

\begin{enumerate}[label=(\roman*), fullwidth] 
\item
Operator norm in $\sL^2$ of $(hD)^k \uptheta(\tau -H)$ does not exceed $C$ for $k=0,1,2$;

\item
Operator norm in $\sL^2$ of 
$(hD)^k\bigl((hD-A)\cdot\boldsigma\bigr) \uptheta(\tau -H)$ does not exceed $C$ for $k=0,1$.
\end{enumerate}
\end{proposition}

\begin{proof}
(i) Let $u =\uptheta(\tau-H) f$. Then 
$\|u\|\le \|f\|$ and as 
\begin{equation}
\|A\|_6 \le C\| \partial A\| \le C(\kappa M)^{\frac{1}{2}}h
\label{1-21}
\end{equation}
we conclude that 
\begin{multline*}
\|hD u\| \le \|(hD-A)u\| +\|A u\| \le 
\|(hD-A)u\| +C\|A \|_6 \|u\|_3\le\\[2pt]
\|(hD-A)u\| +C(\kappa M)^{\frac{1}{2}}h \|u\|^{3/4}\|u\|_6^{1/4}\le\\[2pt]
\|(hD-A)u\| +C\varsigma^{\frac{1}{2}} \|u\|^{3/4}\|h Du\|^{1/4}\le\quad \\[2pt]
\|(hD-A)u\| + \frac{1}{2} \|hDu\| + C\varsigma  \|u\|
\end{multline*}
so due to (\ref{1-20})
\begin{equation}
\|hD u\| \le 2\|(hD-A)u\| +C\|u\|
\label{1-22}
\end{equation}
On the other  hand, for $B=\nabla \times A$ and $\tau\le c$
\begin{multline*}
\|(h D-A)u\| ^2 \le C\|u\|^2 + (h |B|u,u)\le C\|u\|^2 + h\|B\|  \|u\|_4^2 \le \\[3pt]
C\|u\|^2 +  C (\kappa M)^{\frac{1}{2}} h^2  \|u\|  \cdot \|u\|_6 \le 
C\|u\|^2 + C (\kappa M)^{\frac{1}{2}}h  \|u\| \| h D u\|
\end{multline*}
and due to (\ref{1-22}) we conclude that 
\begin{equation}
\|hDu\| + \|(h D-A)u\|  \le C\|u\|.
\label{1-23}
\end{equation}
So, for $j=0,1$  statement (i) is proven. 

Further, as $h^2D^2=(hD-A)^2 + A(hD-A) +A hD -h[D,A]$ we in the same way as before (and using (\ref{1-23}) conclude that 
\begin{gather*}
\|h^2D^2u\| 
\le C\|u\|^2 +  \frac{1}{4} \|hD (hD-A)u\|+ \frac{1}{4} \| h^2D^2u\| \\
\shortintertext{and therefore}
\|h^2D^2u\|\le C\|u\|^2 + C \|A hDu \| 
\end{gather*}
and repeating the same arguments we get  $\|h^2D^2 u\| \le C\|u\|$; so for $j=2$ statement (i) is proven. 

\medskip\noindent
(ii) Statement (ii) is proven in the same way. 
\end{proof}

\begin{corollary}\label{cor-1-7}
Let \textup{(\ref{1-19})} and \textup{(\ref{1-20})} be fulfilled. Then  as 
$\tau\le c$
\begin{equation}
e(x,x,\tau) \le Ch^{-3},\qquad 
|\bigl((hD-A)\cdot \boldsigma )  e(x,y,\tau)|_{x=y} |\le Ch^{-3}.
\label{1-24}
\end{equation}
\end{corollary}
\begin{proof}
Due to proposition \ref{prop-1-6} operator norms from $\sL^2$ to $\sC$  of both $\uptheta (\tau -H)$ and $\bigl((hD-A)\cdot \boldsigma )\uptheta (\tau -H)$ do not exceed $C$ and the same is true for an adjoint operator which imply both claims.
\end{proof}

\begin{corollary}\label{cor-1-8}
Let \textup{(\ref{1-19})} and \textup{(\ref{1-20})} be fulfilled and $A$ be a minimizer. Then for arbitrarily small exponent $\delta>0$
\begin{gather}
\|\partial A\|_{\sC^{1-\delta}} \le C\kappa h^{-1},\label{1-25}\\
\|\partial  A\|_{\infty}\le h^{-\frac{4}{5}-\delta}\label{1-26}
\end{gather}
where $\sC^\theta$ is the scale of H\"older spaces and $\delta >0$ is arbitrarily small.
\end{corollary}

\begin{proof}
Really, due to (\ref{1-13})  for a minimizer 
$\|\Delta A\|_{\infty}\le C\kappa h^{-1}$. Also we know that 
$\|\partial A\| \le C(\kappa Mh^2)^{\frac{1}{2}}\le Ch^{\frac{1}{2}}$ due to (\ref{1-20}). Then (\ref{1-25}) holds due to the standard properties of the elliptic equations\footnote{\label{foot-2} Actually we can slightly improve this statement.}.

Therefore if at some point $y$ we have $|\partial A(y)|\gtrsim \mu$, it is true in its $\epsilon  (\mu h \kappa ^{-1} )^{1-\delta}$-vicinity (provided $\mu \le \kappa h^{-1}$) and then  
\begin{gather*}
\|\partial A\|^2 \gtrsim 
\mu ^2 (\mu h \kappa^{-1})^{3(1-\delta)}\\
\intertext{and  we conclude that} 
\mu ^2 (\mu h \kappa^{-1})^{3(1-\delta)} \le C\kappa h^2 M \iff 
\mu^{5-3\delta} \le C\kappa^{4-3\delta}h^{-1+3\delta}M
\end{gather*}
and one can see easily that (\ref{1-26}) holds due to (\ref{1-20}) and $h^{-1}\le M\le h^{-3}$.  

On the other hand, if $\mu \ge \kappa h^{-1}$ then we need to take $\epsilon$-vicinity and then $\mu^2 \le C\kappa Mh^2 \le C h^{\frac{1}{2}}$ where we used (\ref{1-20}) again. Then (\ref{1-26}) is proven.
\end{proof}

\begin{remark}\label{rem-1-9}
\begin{enumerate}[label=(\roman*), fullwidth]
\item It is not clear if it is possible to generalize this theory to arbitrary $d\ge 2$ with magnetic field energy is given by
\begin{equation}
\frac{1}{\kappa h^{d-1}}\int \bigl(|\partial A|^2-|\nabla\cdot A|^2\bigr) \,dx 
\label{1-27}
\end{equation}
Surely one should use generalized Pauli matrices $\upsigma_j$ in the definition of the operator. Especially problematic are $d\ge 5$.
\item Therefore while arguments  of section~\ref{sect-2} below remain valid for $d\ne 3$, so far they remain conditional (if a minimizer exists and satisfies some crude estimates).
\end{enumerate}
\end{remark}

\chapter{Microlocal analysis unleashed}
\label{sect-2}

\section{Sharp estimates}
\label{sect-2-1}

Now we can unleash the full power of microlocal analysis but we need to extend it to our framework. It follows by induction from (\ref{1-25})--(\ref{1-26}) and the arguments we used to derive these estimates that 
\begin{equation}
\|\partial A\|_{\sC^{n-\delta}} \le C_n\kappa h^{-1-n},
\label{2-1}
\end{equation}
so $A$ is ``smooth'' in $\varepsilon = h$ scale while for rough microlocal analysis as in \cite{bronstein:ivrii:IRO1} and section~\ref{book_new-sect-2-3} of \cite{futurebook} one needs $\varepsilon = Ch|\log h|$ at least. We consider in this section arbitrary $d\ge 2$; see however remark~\ref{rem-1-9}.

Note that 
\begin{claim}\label{2-2}
For a commutator of a pseudo-differential operator with a smooth symbol and $\sC^{\theta+1}$ function $A$ a usual commutator formula holds modulo 
$O( h^{\theta+1}\|\partial A\|_{\sC^{\theta}})$ for any non-integer $\theta >0$.
\end{claim}

\begin{proposition}\label{prop-2-1}
Assume that 
\begin{gather}
\mu \Def \|\partial  A\|_{\infty, B(x,1)} \le C_0\label{2-3},\\
\shortintertext{and in in $B(x,1)$}
|V|\ge \epsilon_0.\label{2-4}
\end{gather}
Then for $|\alpha|\le 2$, $|\beta|\le 2$, $\theta >1$
\begin{multline}
|F_{t\to h^{-1}\tau} \chi_T(t) 
\bigl((hD_x)^\alpha (hD_y)^\beta U(x,y,t)\bigr)\bigr|_{x=y}|\le \\
Ch^{1-d+s} T^{-s}+ C h^{-d+\theta}T^2  \|\partial  A\|_{\sC^{\theta}}
\label{2-5}
\end{multline}
where $\chi \in \sC^\infty ([-1,-\frac{1}{2}]\cup [\frac{1}{2},1])$, $\chi_T(t)=\chi(t/T)$, $Ch \le T\le \epsilon $, $|\tau|\le \epsilon$ and $U(x,y,t)$ is the Schwartz kernel of $e^{ih^{-1}tA}$.
\end{proposition}

\begin{proof}
Consider first $T\asymp 1$. First, using standard propagation arguments as in \cite{bronstein:ivrii:IRO1}  and section~\ref{book_new-sect-2-3} of \cite{futurebook} one can prove  the finite propagation speed with respect to $x$, namely that
\begin{equation}
|F_{t\to h^{-1}\tau} \chi_T(t) (hD_x)^\alpha(hD_y)^\beta U| \le C h^s
\qquad \text{as\ \ } |x-y|\ge C_0T ,\; \tau\le c_0
\label{2-6}
\end{equation}
with an arbitrarily large exponent $s$. Further, using (\ref{2-6}) we prove  the finite propagation speed with respect to $\xi$: 
\begin{multline}
|F_{t\to h^{-1}\tau} \chi_T(t) (hD_x)^\alpha(hD_y)^\beta 
\varphi_1 (hD_x)\varphi_2 (hD_y) U| \le 
C h^s+ Ch^{\theta}\|\partial A\|_{\sC^\theta}\\[3pt]
 \text{as\ \ }  \varphi_1,\varphi_2  \in \sC_0^\infty, \ 
\dist(\supp \varphi_1, \supp \varphi_1)\ge C_0T,\  \tau\le c_0
\label{2-7}
\end{multline}
where the last term in the right-hand expression is due to the non-smoothness of $A$.

Furthermore, using (\ref{2-6}) and (\ref{2-7}) we prove that as (\ref{2-4}) is fulfilled there is a propagation:
\begin{multline}
|F_{t\to h^{-1}\tau} \chi_T(t) (hD_x)^\alpha(hD_y)^\beta U| \le 
C h^s+ Ch^{\theta}\|\partial A\|_{\sC^\theta}\\[3pt]
 \text{as\ \ } |x-y|\le \epsilon_0 T ,\; |\tau|\le \epsilon, 
T\le \epsilon 
\label{2-8}
\end{multline} 
which implies (\ref{2-5}) under additional assumption $T\asymp 1$; here the last term in the right-hand expression is inherited from (\ref{2-7}). 

Finally, for $Ch \le T \le \epsilon$ we use rescaling $t\mapsto t/T$, $x\mapsto x/T$, $h\mapsto h/T$ and we use ``homogenized'' $\sC^\theta$ norms.  Detailed proof for   will be published later. 
\end{proof}

\begin{corollary}\label{cor-2-2}
In the framework of proposition~\ref{prop-2-1} as 
$|\alpha|\le 2$,  $|\beta|\le 2$
\begin{multline}
|F_{t\to h^{-1}\tau} \bigl[\bar{\chi}_T(t) 
\bigl((hD_x)^\alpha (hD_y)^\beta U(x,y,t)\bigr)\bigr]\bigr|_{x=y}|\le \\
Ch^{1-d}+ C T^2 h^{-d+\theta}  \|\partial  A\|_{\sC^{\theta}}  
\label{2-9}
\end{multline}
where $\bar{\chi}\in \sC^\infty ([-1,1])$ and
\begin{multline}
|\bigl[\bigl((hD_x -A(x))\cdot\boldsigma\bigr)^\alpha 
\bigl((hD_y-A(y))\cdot\boldsigma\bigr)^\beta e(x,y,t)\bigr]\bigr|_{x=y} -\\[3pt]
\Weyl_{\alpha,\beta} (x)|\le  
Ch^{1-d}+ C h^{-d+\frac{1}{2}(\theta+1)}  
\|\partial  A\|_{\sC^{\theta}}  ^{\frac{1}{2}}
\label{2-10}
\end{multline}
where 
\begin{gather}
\Weyl_{\alpha,\beta}(x) = \const \, h^{-d} 
\int_{\{H(x,\xi) \le \tau\}}
 \bigl((\xi -A(x))\cdot\boldsigma \bigr)^{\alpha+\beta}\, d\xi
\label{2-11}\\
\intertext{is the corresponding Weyl expression and}
H(x,\xi)= \bigl(\xi -A(x)\cdot \boldsigma \bigr)^2-V(x);
\label{2-12}
\end{gather}
in particular Weyl  is $0$ as $|\alpha|+|\beta|=1$.
\end{corollary}

\begin{proof}
Obviously summation of (\ref{2-5}) over $C_0h \le |t|\le T$ and a trivial estimate by $Ch^{1-d}$ of the contribution of the interval $|t|\le C_0 h$ implies (\ref{2-9}). 

Then the standard Tauberian arguments and (\ref{2-9}) imply that  (\ref{2-10}) would be correct if we used Tauberian expression with $T=T^*$ instead of $\Weyl$ where we pick up 
\begin{equation}
T^*= \epsilon \min \bigl(1, 
h^{-\frac{1}{2}(\theta-1)}\|\partial A\|_{\sC^\theta}^{-\frac{1}{2}}\bigr).
\label{2-13}
\end{equation}
Meanwhile the Tauberian formula and (\ref{2-5}) imply that the contribution of an interval $\{t: |t|\asymp T\}$ with $ h\le T\le T^*$ to the Tauberian expression does not exceed the right-hand expression of (\ref{2-5}) divided by $T$, i.e.
\begin{equation*}
Ch^{1-d+s}T^{-s-1} + C T h^{-d+\theta}  \|\partial  A\|_{\sC^{\theta}};
\end{equation*}
summation over $T_*\Def h^{1-\delta}\le T\le T^*$ results in  the right-hand expression of (\ref{2-10}). 

So, we need to calculate only the contribution of $\{t:|t|\le T_*\}$ but one can see easily that modulo indicated error  it coincides with $\Weyl_{\alpha,\beta}$.
\end{proof}

\begin{corollary}\label{cor-2-3}
\begin{enumerate}[label=(\roman*), fullwidth]
\item If assumption \textup{(\ref{2-3})} is replaced by 
\begin{equation}
\mu \Def \|\partial  A\|_{\infty}\le C h^{-1+\sigma}
\label{2-14}
\end{equation}
with $\sigma> 0$ then 
\begin{multline}
|\bigl[\bigl((hD_x -A(x))\cdot\boldsigma\bigr)^\alpha 
\bigl((hD_y-A(y))\cdot\boldsigma\bigr)^\beta e(x,y,t)\bigr]\bigr|_{x=y} -\\[3pt]
\Weyl_{\alpha,\beta} (x)|\le  
C\bar{\mu} h^{1-d}+ C \bar{\mu} ^{-\frac{1}{2}}h^{-d+\frac{1}{2}(\theta+1)}  \|\partial  A\|_{\sC^{\theta}}  ^{\frac{1}{2}}  
\label{2-15}
\end{multline}
and
\begin{equation}
T^*= \epsilon \min \bigl(\bar{\mu}^{-1}, 
\bar{\mu} h^{-\frac{1}{2}(\theta-1)}
\|\partial A\|_{\sC^\theta}^{-\frac{1}{2}}\bigr)
\label{2-16}
\end{equation}
with $\bar{\mu}=\max(\mu,1)$;
\item As $d\ge 3$ one can skip assumption \textup{(\ref{2-4})}.
\end{enumerate}
\end{corollary}
\begin{proof}
As $d\ge 3$ (ii) is proven by the standard rescaling technique: $x\mapsto x \ell$, $h\mapsto \hbar= h\ell^{-\frac{3}{2}}$, $A\mapsto A \ell^{-\frac{1}{2}}$ with $\ell = \max( \epsilon |V|, h^{\frac{2}{3}})$ (see \cite{ivrii:IRO2} and chapter~\ref{book_new-sect-5} of \cite{futurebook}. 

As $\mu \ge 1$ (i) is proven by the standard  rescaling technique 
$x\mapsto \mu x$, $h\mapsto \hbar=h\mu$.
\end{proof}

\begin{proposition}\label{prop-2-4}
Let $\kappa\le c$,  \textup{(\ref{2-14})} be fulfilled,  and let $A$ be a minimizer.  As $d=2$ let \textup{(\ref{2-4})} be also fulfilled.  Then as $\theta\in (1,2)$
\begin{gather}
\|\partial  A\|_{\sC^{\theta-1}} + h^{\theta-1} \|\partial  A\|_{\sC^{\theta}} \le C\kappa +  C\|\partial A\|'\label{2-17}\\
\shortintertext{with}
\|\partial A\|'\Def\sup _y \|\partial A\|_{\sL^2(B(y,1))}.
\label{2-18}
\end{gather}
\end{proposition}

\begin{proof}
Consider expression for $\Delta A$. According to (\ref{1-13}) and (\ref{2-15}) we get then for any $\theta \in (1,2)$
\begin{gather}
|\Delta A|+ |h\partial \Delta A| \le 
C\kappa \bigl(  \bar{\mu} +
\bar{\mu}^{-\frac{1}{2}} h^{\frac{1}{2}(\theta-1)}
\|\partial  A\|_{\sC^{\theta}}^{\frac{1}{2}}  \bigr),\label{2-19}\\
\intertext{which implies that for any $\theta '\in (1,2)$} 
h^{\theta'-1} \|\partial A\|_{\sC^{\theta'} }
 \le C\kappa \bigl(  \bar{\mu} +
\bar{\mu}^{-\frac{1}{2}} h^{\frac{1}{2}(\theta-1)}
\|\partial  A\|_{\sC^{\theta}}^{\frac{1}{2}}  \bigr)+C\mu  ,\notag
\end{gather}
and picking up $\theta'=\theta$ we conclude that  
\begin{equation}
h^{\theta-1} \|\partial A\|_{\sC^{\theta}} 
 \le C \bigl( \kappa \bar{\mu}  + \kappa^2 \bar{\mu}^{-1})+C\mu.
 \label{2-20}
\end{equation}
As $\kappa \le c$ the right-hand expression does not exceed $C(\kappa+\mu) $; then the right-hand expression in (\ref{2-19}) also does not exceed $C(\kappa+\mu) $; then $\|\partial A\|_{\sC^{\theta -1}}\le C(\kappa+\mu)$ and then $\mu \le C\kappa + C\|\partial A\|'$, which implies (\ref{2-17}).
\end{proof}

Having this strong estimate to $A$ allows us to prove 

\begin{theorem}\label{thm-2-5}
Let $\kappa\le c$,  \textup{(\ref{2-14})} be fulfilled, and let $d= 3$. Then
\begin{gather}
\E^*= \Weyl_1 +O(h^{-1})\label{2-21}\\
\intertext{and a minimizer $A$ satisfies} 
\|\partial A\|\le C \kappa^{\frac{1}{2}}h^{\frac{1}{2}}\label{2-22}\\
\shortintertext{and}
\|\partial A\|_{\sC^{\theta-1}} + h^{\theta-1}\|\partial A\|_{\sC^{\theta}}\le 
C \kappa^{\frac{1}{2}}h^{\frac{1}{2}} +C\kappa.
\label{2-23}
\end{gather}
\end{theorem}

\begin{proof}
In virtue of (\ref{2-9}) and (\ref{2-17}) the Tauberian error with $T\asymp 1$ when calculating $\Tr H^-_{a,V}$ does not exceed 
\begin{equation}
C\bar{\mu}^2 h^{2-d}+ C(\kappa +\|\partial A\|') h^{2-d}.
\label{2-24}
\end{equation}
We claim that 
\begin{claim}\label{2-25}
Weyl error when calculating $\Tr H^-_{a,V}$ also does not exceed (\ref{2-24}). 
\end{claim}
Then 
\begin{multline}
\E^*(A) \ge \Weyl_1 - C\bar{\mu}^2 h^{2-d}- C(\kappa +\|\partial A\|') h^{2-d} +
\kappa^{-1} h^{1-d} \|\partial A\|^2 \ge \\
\Weyl_1 - Ch^{2-d} +  \frac{1}{2\kappa } h^{1-d} \|\partial A\|^2
\label{2-26}
\end{multline}
because $\bar{\mu}\le C\|\partial A\| +1$ due to (\ref{2-17}). This implies an estimate of $\E^*$ from below and combining with the estimate $\E^* \le \E^*(0) =\Weyl_1 + Ch^{2-d}$ from above we arrive to (\ref{2-21}) and (\ref{2-22}) and then (\ref{2-23}) due to (\ref{2-17}).

\bigskip
To prove (\ref{2-25}) let us plug $A_\varepsilon$ instead of $A$ into $e_1(x,x,0)$. Then in virtue of rough microlocal analysis contribution of $\{t: T_* \le |t|\le \epsilon\}$ with $T_*=h^{1-\delta}$ would be negligible and contribution of $\{t: |t|\le T_*\}$ would be $\Weyl_1 + O(h^{2-d})$.

Let us calculate an error which we made plugging $A_\varepsilon$ instead of $A$ into $e_1(x,x,0)$. Obviously it does not exceed 
$Ch^{-d}\|A-A_\varepsilon\|_\infty $ and since 
$\|A-A_\varepsilon\|_\infty \le 
C\varepsilon ^{\theta+1}\|\partial A\|_{\sC^\theta}$ this error does not exceed 
$Ch^{\theta+1-d-4\delta}\|\partial A\|_{\sC^\theta}$ which is marginally worse than what we are looking for. However it is good enough to recover a weaker version of (\ref{2-21}) and (\ref{2-22}) with an extra factor $h^{-\delta_1}$ in their right-hand expressions. Then (\ref{2-17}) implies a bit  weaker version of (\ref{2-23}) and in particular that its left-hand expression does not exceed $C$.

Knowing this let us consider the two term approximation. With the above knowledge one can prove easily that the error in two term approximation does not exceed $Ch^{3-d -\delta'}$ with $\delta '= 100\delta$.

Then the second term in the Tauberian expression is 
\begin{equation}
\int \bigl((H_{A,V}-H_{A_\varepsilon,V})e^\T_{(\varepsilon)}(x,y,0)\bigr)
\bigr|_{y=x}\,dx.
\label{2-27}
\end{equation}
where subscript means that we plugged $A_\varepsilon$ instead of  $A$ and superscript $^\T$ means that we consider Tauberian expression with $T=T^*=\epsilon$. But then the contribution of $\{t: T_*\le |t|\le T^*\}$ is also negligible and modulo $Ch^{\theta+2-d-4\delta}\|\partial A\|_{\sC^\theta}$ we get a Weyl expression. However 
\begin{equation}
(H_{A,V}-H_{A_\varepsilon,V}) = -2(\xi -A_\varepsilon)\cdot (A-A_\varepsilon)+ 
|A-A_\varepsilon|^2
\label{2-28}
\end{equation} 
and the first term kills Weyl expression as integrand is odd with respect to $(\xi -A_\varepsilon)$ while the second as one can see easily makes it smaller than $Ch^{3-d -\delta'}$. Therefore (\ref{2-25}) has been proven.
\end{proof}

\begin{remark}\label{rem-2-6}
\begin{enumerate}[fullwidth, label=(\roman*)]
\item For $d=2$ we cannot skip (\ref{2-4}) at the stage we did it for $d\ge 3$. However results of the next section allow us to cure this problem using partition-and-rescaling technique.

\item Actually  we have an estimate 
\begin{equation}
|\partial  A(x)-\partial A(y)|\le C \kappa |x-y| (|\log |x-y||+1) + C\mu.
\label{2-29}
\end{equation}
Combining with (\ref{2-22}) we conclude that 
\begin{equation}
\|\partial A\|_{\infty} \le C \kappa^{(d+1)/(d+2)}|\log h|^{d/(d+2)} h^{1/(d+2)}
\label{2-30}
\end{equation}
\end{enumerate}
\end{remark} 

\section{Classical dynamics and sharper estimates}
\label{sect-2-2}

Now we want to improve remainder estimate $O(h^{2-d})$ to $o(h^{2-d})$. Sure, we need to impose condition to the classical dynamical system and as 
$|\partial A |=O(h^\sigma)$ with $\sigma>0$ due to (\ref{2-30}) it should be dynamical system associated with Hamiltonian flow generated by $H_{0,V}$:

\begin{claim}\label{2-31}
The set of periodic points of the dynamical system associated with Hamiltonian flow generated by $H_{0,V}$ has measure $0$ on the energy level $0$.
\end{claim}
Recall that on $\{(x,\xi): H_{0,V}(x,\xi)=\tau\}$ a natural density $d\upmu_\tau= dxd\xi :dH|_{H=\tau}$ is defined.

The problem is we do not have quantum propagation theory for $H_{A,V}$ as $A$ is not a ``rough'' function. However it is rather regular function, almost $\sC^2$, and $(A-A_\varepsilon)$ is rather small: $|A-A_\varepsilon|\le \eta= Ch^{2-3\delta}$ and $|\partial (A-A_\varepsilon)|\le Ch^{1-3\delta}$ and therefore we can apply a method of successive approximations with the unperturbed operator $H_{A_\varepsilon, V}$ as long as $\eta T/h\le h^\sigma$ i.e. as $T\le h^{1-4\delta}$. Here we however have no use for such large $T$ and consider $T=O(h^{-\delta})$. 

Consider 
\begin{equation}
F_{t\to h^{-1} \tau} \chi_T(t) U(x,y,t),
\label{2-32}
\end{equation}
and consider terms of successive approximations. Then if we forget about microhyperbolicity arguments the first term will be $O(h^{-d}T)$, the second $O(h^{-1-d}\eta T^2)= O(h^{1-d-\delta'})$ and the error 
$O(h^{-2-d}\eta^2 T^3)=O(h^{2-d-\delta''})$. Therefore as our goal is $O(h^{1-d})$ we need to consider the first two terms only. The first term is the same expression (\ref{2-32}) with $U$ replaced by $U_{(\varepsilon)}$. 

Consider the second term, it corresponds to
\begin{gather}
U'_{(\varepsilon)} = \bigl[ i h^{-1}\int _0^t 
e^{i(t-t') h^{-1}H_{A_\varepsilon,V} } \bigl(H_{A,V}-H_{A_\varepsilon,V}\bigr) e^{it' h^{-1}H_{A_\varepsilon,V} } \,dt' \bigr]
\label{2-33}\\
\shortintertext{and then}
\Tr \bigl(e^{ih^{-1}tA} \psi \bigr)= ih^{-1} \Tr \Bigl(\bigl(H_{A,V}-H_{A_\varepsilon,V}\bigr) 
  e^{ih^{-1}tH_{A_\varepsilon,V}} \psi_t \Bigr)\label{2-34}\\
\shortintertext{with}
\psi_t = \int _0^t e^{ih^{-1}t'H_{A_\varepsilon,V}}\psi e^{-ih^{-1}t'H_{A_\varepsilon,V}}\,dt'.\notag
\end{gather}
Here $[S](x,y)$ denotes the Schwartz kernel of operator $S$. We claim that
\begin{equation}
|F_{t\to h^{-1}\tau}\chi_T(t) \Tr U'_{(\varepsilon)} \psi |\le C\eta T^2 h^{-d}.
\label{2-35}
\end{equation}
Here in comparison with the trivial estimate we gained factor $h$. The proof of (\ref{2-35}) can be done easily by the standard rough microlocal analysis arguments and we will provide a detailed proof later. 

Therefore we arrive to the estimate
\begin{equation}
|F_{t\to h^{-1} \tau} \chi_T(t) 
\Tr \bigl(\bigl(e^{it' h^{-1}H_{A,V} } - e^{it' h^{-1}H_{A_\varepsilon,V} } \bigr ) \psi  \bigr)|\le Ch^{1-d}.
\label{2-36}
\end{equation}
On the other hand traditional methods imply that as $d\ge 3$
\begin{equation}
|F_{t\to h^{-1} \tau} \chi_T(t) 
\Tr  \bigl(e^{it' h^{-1}H_{A_\varepsilon,V} } \psi  \bigr)|\le 
Ch^{1-d} T \upmu (\Pi_{T, \rho}) + C_{T,\rho} h^{1-d+\delta}
\label{2-37}
\end{equation}
where $\Pi _T$ is the set of  points on energy level $0$, periodic with periods not exceeding $T$,  $\Pi _{T,\rho}$ is its $\rho$-vicinity, $\rho>0$ is arbitrarily small. 

\begin{theorem}\label{thm-2-7}
Let $\kappa\le c$,  \textup{(\ref{2-14})} be fulfilled, and let $d=3$. Furthermore, let condition \textup{(\ref{2-31})} be fulfilled (i.e. 
$\upmu _0(\Pi_\infty)=0$). Then 
\begin{gather}
\E^*= \Weyl^*_1 +o(h^{-1})\label{2-38}\\
\shortintertext{where}
\Weyl^*_1 =\Weyl_1 +\varkappa h^{-1}\int V_+^{\frac{3}{2}}\Delta V  \,dx
\label{2-39} 
\end{gather}
calculated in the standard way for $H_{0,V}$ and a minimizer $A$ satisfies similarly improved versions of \textup{(\ref{2-22})} and \textup{(\ref{2-23})}.
\end{theorem}

\begin{remark}\label{rem-2-8}
\begin{enumerate}[label=(\roman*), fullwidth]
\item Under stronger assumptions to the Hamiltonian flow one can recover better estimates like  $O(h^{2-d}|\log h|^{-2})$ or even $O(h^{2+\delta-d})$ (like in  subsubsection~\ref{book_new-sect-4-4-4-3}.3 of \cite{futurebook}).

\item We leave to the reader to calculate the numerical constants $\varkappa_*$ here and in (\ref{3-4}), $\varkappa=\varkappa_1-\frac{2}{3}\varkappa_2$.
\end{enumerate}
\end{remark}

\chapter{Local theory}
\label{sect-3}

\section{Localization}
\label{sect-3-1}

The results of the previous section have two shortcomings: first, they impose the excessive requirement to $\kappa$; second, they are not local. However curing the second shortcoming we make the way to addressing the first one as well using the partition and rescaling technique. 

We localize the first term in $\E(A)$ by using the same localization as in \cite{EFS1}: namely we take $\Tr^- (\psi H\psi )$ where 
$\psi\in \sC^\infty_0(B(0,\frac{1}{2}))$, $0\le \psi \le 1$ and some other  conditions will be imposed to it later. Note that 
\begin{equation}
\Tr^-(\psi H\psi ) \ge \int e_1(x,x,0) \psi^2(x)\,dx.
\label{3-1}
\end{equation}
Really, operator $H=H \uptheta(-H) + H(1-\uptheta(-H))$ where $\uptheta (\tau-H)$ is a spectral projector of $H$ and therefore in the operator sense 
$H\ge H^-\Def HE(0)$ and $\psi H\psi \ge \psi H^-\psi$ and therefore all negative eigenvalues of  $\psi H\psi$ are greater than or equal to eigenvalues of the negative operator $\psi H^-\psi$ and then 
\begin{equation*}
\Tr^-(\psi H\psi)\ge \Tr \psi H ^- \psi= \Tr \int^0_{-\infty} \tau d_\tau E(\tau) \psi^2
\end{equation*}
which is exactly the right-hand expression of (\ref{3-1}).

\begin{remark}\label{rem-3-1}
\begin{enumerate}[label=(\roman*), fullwidth]
\item
The right-hand expression of (\ref{3-1}) is an another way to localize operator trace. Each approach has its own advantages. In particular, no need to localize $A$ (see (ii)) and the fact that proposition~\ref{prop-1-5} obviously remains true (due to corollary~\ref{cor-2-3}) are advantages of 
$\Tr^-(\psi H\psi)$-localization. 

\item As $\Tr^-(\psi H\psi )$ does not depend on $A$ outside of $B(0,\frac{3}{4})$ we may assume that $A=0$ outside of $B(0,1)$. Really, we can always subtract a constant from $A$ without affecting traces and also cut-off $A$ outside of $B(0,1)$ in a way such that $A'=A$ in $B(0,\frac{3}{4})$ and 
$\|\partial A'\| \le c\|\partial A\|_{B(0,1)}$; the price is to multiply $\kappa$ by $c^{-1}$ -- as long as principal parts of asymptotics coincide. 

\item Additivity rather than sub-additivity (\ref{4-2}) and the trivial estimate from the above are advantages of $\Tr \psi H ^-\psi$-localization. It may happen that the latter definition is more useful in applications to theory of heavy atoms and molecules and we will need to recover our results under it.
\end{enumerate}
\end{remark}

Let us estimate from the above:

\begin{proposition}\label{prop-3-2}
Let $\ell(x)$ be a scaling function\footnote{\label{foot-3} I.e. $\ell \ge 0$ and $|\partial \ell|\le \frac{1}{2}$.} and $\psi$ be a function such that  $|\partial^\alpha \psi|\le c\psi \ell^{-2|\alpha|}$ for all 
$\alpha:|\alpha |\le 2$ and $|\psi |\le c\ell^3$\,\footnote{\label{foot-4} Such compactly supported functions obviously exist.}.

Then, as $A=0$,
\begin{gather}
\Tr^- (\psi H\psi) = \int \Weyl_1(x)\psi^2(x)\,dx + O(h^{-1}) 
\label{3-2}\\
\intertext{and under assumption \textup{(\ref{2-31})}}
\Tr^- (\psi H\psi) = \int \Weyl^*_1(x)\psi^2(x)\,dx + o(h^{-1}) 
\label{3-3}\\
\shortintertext{with}
\Weyl^*_1(x)= \Weyl_1(x) + \varkappa_1 h^{-1}V_+^{\frac{3}{2}}\Delta V + 
\varkappa_2 h^{-1}V_+^{\frac{1}{2}}|\nabla V|^2\label{3-4}
\end{gather}
calculated in the standard way for $H_{0,V}$.
\end{proposition}

\begin{proof}
Let us consider $\tilde{H}=\psi H\psi $ as a Hamiltonian and let $\tilde{e}(x,y,\tau)$ be the Schwartz kernel of its spectral projector. Then 
\begin{equation}
\Tr^- (\psi H\psi)= \int  \tilde{e}_1 (x,x,0)\,dx =
\sum_j \int  \tilde{e}_1 (x,x,0)\psi_j^2\,dx
\label{3-5}
\end{equation}
where $\psi_j^2$ form a partition of unity in $\bR^3$ and we need to calculate the right hand expression.

\begin{enumerate}[label=(\roman*), fullwidth]
\item Consider first an $\epsilon \ell$-admissible partition of unity in $B(0,1)$. Let us consider $\gamma$-scale in such element where 
$\gamma = \max(\epsilon \ell^2 , h)$ and we will use $1$ scale in $\xi$.  Then after rescaling $x\mapsto x/\gamma$ semiclassical parameter rescales 
$h\mapsto h_\new = h/\gamma$ and the contribution of each $\gamma$-element to a semiclassical remainder  does not exceed $C\varrho (h/\gamma)^{-1}$ with $\varrho \le \ell^4$ having the same magnitude over element as 
$\gamma \ge 2h$. Then contribution of $\ell$ element to a semiclassical error does not exceed $C\varrho (h/\gamma)^{-1}\times \ell^3 \gamma^{-3} 
\asymp Ch^{-1} \varrho \ell^3 \ell^{-4}\le Ch^{-1}\ell^5$.

As $\ell^2 \asymp h$ the same arguments work with $\ell$ replaced by $h^{\frac{1}{2}}$ and $\gamma = h$ and effective semiclassical parameter $1$.

Therefore the total contribution of all partition elements in $B(0,1)$ to a semiclassical error does not exceed $Ch^{-1}$. 

\item 
However we need to consider contribution of the rest of $\bR^3$. Here we use $\gamma = \frac{1}{2}|x|$, $1$-scale with respect to $\xi$ and take 
$\varrho = h\gamma^{-2}$; then contribution of $\gamma$-element to a semiclassical error does not exceed $C\varrho (h/\gamma)^{-1}\le C\gamma^{-1}$ and summation over partition results in $C$. Thus (\ref{3-2}) is proven.

\item Note that contribution of zone $\|\{x:\ell \le \eta\}$ to the remainder does not exceed $C\ell^2 h^{-1}$; applying in zone $\|\{x:\ell \ge \eta\}$ sharp asymptotics under assumption (\ref{2-31}) we prove (\ref{3-3}).
\end{enumerate}
\end{proof}

\begin{corollary}\label{cor-3-3}
In the framework of proposition~\ref{prop-3-2}(i), (ii) 
\begin{gather}
\E_\psi ^* \Def  \inf_{A} \E_\psi (A) \le \Weyl_1 + Ch^{-1}\label{3-6}\\
\shortintertext{and}
\E_\psi ^*  \le \Weyl^*_1 + Ch^{-1}\label{3-7}\\
\shortintertext{respectively with}
\E _\psi (A)\Def \Tr^- (\psi H\psi)+ \frac{1}{\kappa h^2}\int |\partial A|^2\,dx \label{3-8}.
\end{gather}
\end{corollary}

Really, we just pick $A=0$.

\section{Estimate from below}
\label{sect-3-2}

Now let us estimate $\E_\psi (A)$ from below. We already know that 
\begin{equation}
\E_\psi (A) \ge \int e_1(x,x,0)\psi^2 \,dx + 
\frac{1}{\kappa h^2}\int_{B(0,1)} |\partial A|^2\,dx.
\label{3-9}
\end{equation}
However we need an equation for an optimizer and it would be easier for us to deal with even lesser expression involving $\tau$-regularization. Let us rewrite the first term in the form 
\begin{multline*}
\int^0_{-\infty}  \bar{\varphi}(\tau/L) \tau\, d_\tau e(x,x,\tau) +
\int^0_{-\infty}  (1-\bar{\varphi}(\tau/L)) \tau\, d_\tau e(x,x,\tau)\ge\\
\int ^L_{-\infty}\Bigl( \bar{\varphi} (\tau/L) (\tau-L)\, d_\tau e(x,x,\tau) +
  (1-\bar{\varphi}(\tau/L)) \tau\, d_\tau e(x,x,\tau)\Bigr)
\end{multline*}
where $\bar{\varphi}\in \sC^\infty_0 ([-1,1])$ equals $1$ in $[-\frac{1}{2},\frac{1}{2}]$ and let us estimate from  below
\begin{multline}
\E'_\psi (A)= 
\int  \Bigl(\int^L_{-\infty}  \bar{\varphi}(\tau/L) (\tau-L) d_\tau e(x,x,\tau)(x)+\\
  (1-\bar{\varphi}(\tau/L)) (\tau-L)\, d_\tau e(x,x,\tau)\Bigr)\psi (x)\,dx+\\
\frac{1}{kh^2}\int_{B(0,1)} |\partial A|^2 (\varrho(x)+K^{-1})\,dx.
\label{3-10}
\end{multline}

\begin{proposition}\label{prop-3-4}
Let $A$ be a minimizer of $\E'_\psi (A)$. Then 
\begin{multline}
\frac{2}{\kappa h^2}\Delta A_j (x)= \Phi_j\Def \\
\bigl(\upsigma_j\upsigma_k (hD_k-A_k)_x  + 
\upsigma_k\upsigma_j (hD_k-A_k)_y \bigr)\times \\
\int ^L_{-\infty}\Bigl[\bar{\varphi} (\tau/L) (\tau-L) \Res_\bR (\tau -H)^{-1} \psi (\tau -H)^{-1}  +\\
(1-\bar{\varphi}(\tau/L)) \tau (\tau-L)\Res_\bR (\tau -H)^{-1} 
\psi (\tau -H)^{-1} \Bigr](x,y)\,d\tau \Bigr|_{y=x}
\label{3-11}
\end{multline}
where again $[S](x,y)$ is the Schwartz kernel of $S$.
\end{proposition}
  
\begin{proof}
Follows immediately from the proof of proposition~\ref{prop-1-4}.
\end{proof}
  
\begin{proposition}\label{prop-3-5}
Let \textup{(\ref{1-19})} and \textup{(\ref{1-20})} be fulfilled. Then as $\tau\le c$
\begin{enumerate}[label=(\roman*), fullwidth] 
\item
Operator norm in $\sL^2$ of $(hD)^k (\tau -H)^{-1}$ does not exceed 
$C|\Im \tau|^{-1}$ for $k=0,1,2$;

\item
Operator norm in $\sL^2$ of 
$(hD)^2\bigl((hD-A)\cdot\boldsigma\bigr) (\tau -H)^{-1}$ does not exceed $C|\Im \tau|^{-1}$ for $k=0,1,2$.
\end{enumerate}
\end{proposition}

\begin{proof}
Proof follows the same scheme as the proof of proposition \ref{prop-1-6}.
\end{proof}

\begin{proposition}\label{prop-3-6}
Let \textup{(\ref{1-19})} and \textup{(\ref{1-20})} be fulfilled. Then
$|\Phi (x)|\le Ch^{-3}$.
\end{proposition}

\begin{proof}
Let us  estimate 
\begin{equation}
|\int \tau \varphi (\tau/L)  \Res_\bR 
\Bigl[ T (\tau -H)^{-1} \psi (\tau -H)^{-1} \Bigr] (x,y)\,d\tau|
\label{3-12}
\end{equation}
where  $L\le c$ and $\varphi \in \sC^\infty_0 ([-1,1])$  and a similar expression with a factor $(\tau -L)$ instead of $\tau$; here either $T=I$, or $T=(hD_k-A_k)_x$ or  
$T=(hD_k-A_k)_y$. 

Proposition \ref{prop-3-5} implies that the Schwartz kernel of the integrand does not exceed $Ch^{-3}|\Im \tau|^{-2}$ and therefore expression (\ref{3-12}) does not exceed $CL^2 \times h^{-3}L^{-2}= Ch^{-3}$.

Then what comes out in $\Phi$ from the term with the factor 
$\bar{\phi}(\tau /h)$  does not exceed $Ch^{-3}$.

Representing 
$(1-\bar{\phi}(\tau /h))$ as a sum of $\varphi (\tau /L)$ with $L= 2^n h$ with $n=0,\ldots, \lfloor |\log h|\rfloor +c$ we estimate the output of each term by $Ch^{-3}$ and thus the whole sum by $Ch^{-3}|\log h|$. 

\smallskip
To get rid off the logarithmic factor we rewrite 
$(\tau -H)^{-1}\psi (\tau -H)^{-1}$ as $-\partial (\tau -H)^{-1}\psi
+ (\tau -H)^{-2}[h,\psi ](\tau -H)^{-1}$; if we plug only the second part we recover a factor $h/L$ where $h$ comes from the commutator and $1/L$ from the increased singularity; an extra operator factor in the commutator is not essential. Then summation over partition results in $Ch^{-3}$.

Plugging only the first part we do not use the above decomposition but an equality
$\Res_\bR (\tau -H)^{-1}\, d\tau =d_\tau \uptheta (\tau -H)$.
\end{proof}

\begin{corollary}\label{cor-3-7}
Let \textup{(\ref{1-19})} and \textup{(\ref{1-20})} be fulfilled and $A$ be a minimizer. Then \textup{(\ref{1-25})} and \textup{(\ref{1-26})} hold.
\end{corollary}

\begin{proof}
Proof follows the proof of corollary~\ref{cor-1-8}.
\end{proof}

Now we can recover both proposition~\ref{prop-2-4} and finally theorems~\ref{thm-2-5} and~\ref{thm-2-7}:

\begin{theorem}\label{thm-3-8}
Let \textup{(\ref{1-19})} and $\kappa\le c$ be fulfilled. Then 
\begin{enumerate}[label=(\roman*), fullwidth]
\item The following estimate holds:
\begin{equation}
\E^*_\psi= \int \Weyl_1(x)\psi^2(x)\,dx  = O(h^{-1})
\label{3-13}
\end{equation}
and and a minimizer $A$ satisfies \textup{(\ref{2-22})} and \textup{(\ref{2-23})};
\item Furthermore, let assumption \textup{(\ref{2-31})} be fulfilled (i.e. 
$\upmu _0(\Pi_\infty)=0$). Then 
\begin{equation}
\E^*_\psi- \int \Weyl^*_1(x)\psi^2(x)\,dx  =o(h^{-1})\label{3-14}
\end{equation}
and a minimizer $A$ satisfies similarly improved versions of \textup{(\ref{2-22})} and \textup{(\ref{2-23})}.
\end{enumerate}
\end{theorem}

\chapter{Rescaling}
\label{sect-4}

\section{Case \texorpdfstring{$\kappa \le 1$}{\textkappa  \textle 1}}
\label{sect-4-1}

We already have an upper estimate: corollary~\ref{cor-3-3}. Let us prove a lower estimate. Consider an error
\begin{equation}
\Bigl(\int \Weyl_1(x)\psi^2\,dx - \E_{\psi}(A)\Bigr)_+.
\label{4-1}
\end{equation}
Obviously $\Tr^-$ is sub-additive
\begin{equation}
\Tr^-(\sum _j \psi_j H\psi_j)\ge \sum_j \Tr^-(\psi_j H\psi_j)
\label{4-2}
\end{equation}
and therefore so is $\E_\psi (A)$. Then we need to consider each partition element and use a lower  estimate for it. While considering partition we use so called ISM identity
\begin{equation}
\sum _j\psi_j^2=1 \implies H= 
\sum_j \bigl(\psi_j H\psi_j + \frac{1}{2}[[H,\psi_j],\psi_j]\bigr).
\label{4-3}
\end{equation}
In virtue of theorem~\ref{thm-1-1}, from the very beginning we need to consider 
\begin{equation}
M= \kappa^\beta h^{-\frac{3}{2}-\alpha}
\label{4-4}
\end{equation}
with $\alpha=\frac{3}{2}$, $\beta=0$ and $\kappa \le c$. But we need to satisfy precondition (\ref{1-20}) which is then 
\begin{equation}
\kappa ^{\beta+1} h^{-\alpha}\le c.
\label{4-5}
\end{equation}
If (\ref{4-5}) is fulfilled with $\alpha=0$ we conclude that the final error is indeed $O(h^{-1})$ or even $o(h^{-1})$ without any precondition.

Let precondition (\ref{4-5}) fail. Let us use $\gamma$-admissible partition of unity $\psi_i$ with $\psi_i$ satisfying after rescaling assumptions of proposition~\ref{prop-3-2}. 

Note that rescaling $x\mapsto x/\gamma$ results in $h\mapsto h_\new=h/\gamma$ and after rescaling in the new coordinates  $\|\partial  A\|^2$ acquires factor $\gamma^{d-2}$ and thus factor $\kappa^{-1}h^{1-d}$ becomes $\kappa^{-1}h^{1-d}\gamma^{d-2}  =
\kappa_\new^{-1} h_\new^{1-d}$ with $\kappa \mapsto \kappa_\new=\kappa \gamma$.

Then \emph{after rescaling\/} precondition (\ref{4-5}) is satisfied provided \emph{before rescaling\/}  
$\kappa^{\beta+1}h^{-\alpha}\gamma^{\alpha +\beta +1}\le c$. Let us pick up
$\gamma = \kappa^{-(\beta+1)/(\alpha +\beta +1) }h^{\alpha/(\alpha +\beta +1)}$.
Obviously if before rescaling condition (\ref{4-5}) failed, then $h\ll \gamma\le 1$.

But then expression (\ref{4-1}) with $\psi$ replaced by $\psi_j$
does not exceed $C(h/\gamma)^{2-d}$ and the total expression (\ref{4-1})  does not exceed 
$Ch^{-1}\gamma^{-2}= C\kappa ^{\beta'} h^{-\frac{3}{2}-\alpha'}$ with 
\begin{equation*}
\beta'= 2(\beta+1)/(\alpha+\beta+1),\qquad \alpha'=-\frac{1}{2}+2\alpha/(\alpha+\beta+1);
\end{equation*}

So, actually we can pick up $M$ with $\alpha,\beta$ replaced by $\alpha',\beta'$ and we have a precondition (\ref{4-5}) with these new $\alpha',\beta'$ and we do not need an old precondition. Repeating the rescaling procedure again we derive a proper estimate with again weaker precondition etc. 

One can see easily that $\alpha'+\beta'+1=\frac{5}{2}$ and therefore on each step $\alpha+\beta+1=\frac{5}{2}$ and we have recurrent relation for $\alpha'$: $\alpha'=-\frac{1}{2}+\frac{4}{5}\alpha$;
and therefore we have sequence for $\alpha$ which decays and then becomes negative. Precondition (\ref{4-5}) has been removed completely and estimate $M=O(h^{-1})$ has been established. After this under assumption (\ref{2-31}) we can prove even sharper asymptotics. Thus we arrive to 

\begin{theorem}\label{thm-4-1}
Let $d=3$, $V\in \sC^{2,1}$, $\kappa \le c$ and let $\psi$ satisfy assumption of proposition~\ref{prop-3-2}. Then

\begin{enumerate}[label=(\roman*), fullwidth]
\item Asymptotics \textup{(\ref{3-13})} holds;
\item Further, if assumption \textup{(\ref{2-31})} is fulfilled then asymptotics \textup{(\ref{3-13})} holds;
\item  If \textup{(\ref{3-13})} or \textup{(\ref{3-14})} holds for $E_\psi (A)$  (we need only an estimate from below) then 
$\|\partial A\|=O(\kappa h)^{\frac{1}{2}})$ or 
$\|\partial A\|=o(\kappa h)^{\frac{1}{2}})$ respectively.
\end{enumerate}
\end{theorem}

\section{Case \texorpdfstring{$1\le \kappa \le h^{-1}$}{1\textle \textkappa  \textle 1/h}}
\label{sect-4-2}

We can consider even the case $1\le \kappa \le h^{-1}$. The simple rescaling-and-partition arguments with $\gamma=\kappa^{-1}$  lead to the following
\begin{claim}\label{4-6}
As $1\le \kappa \le h^{-1}$ remainder estimate $O(\kappa^2h^{-1})$ holds and for a minimizer $\|\partial A\|^2\le C\kappa^3 h$.
\end{claim}
 However we would like to improve it and, in particular prove that as $\kappa$ is moderately large remainder estimate is $O(h^{-1})$ and even $o(h^{-1})$ under non-periodicity assumption.

\begin{theorem}\label{thm-4-2}
Let $d=3$, $V\in \sC^{2,1}$,  and let $\psi$ satisfy assumptions of proposition~\ref{prop-3-2}. Then
\begin{enumerate}[label=(\roman*), fullwidth]
\item As  
\begin{equation}
\kappa \le \kappa^*_h\Def \epsilon h^{-\frac{1}{4}}|\log h|^{-\frac{3}{4}}
\label{4-7}
\end{equation}
 asymptotics \textup{(\ref{3-13})} holds;
\item Furthermore as $\kappa=o(\kappa^*_h)$ and assumption \textup{(\ref{2-31})} is fulfilled then   asymptotics \textup{(\ref{3-14})} holds;
\item As   $1 \le \kappa \le ch^{-1}$
 the following estimate holds:
\begin{equation}
|E^*_\psi - \int \Weyl_1(x)\psi^2 \,dx |\le 
Ch^{-3} (\kappa h)^{\frac{8}{3}} |\log \kappa h|^2
\label{4-8}
\end{equation}
\end{enumerate}
\end{theorem}

\begin{proof}
(i) From (\ref{2-20}) we conclude as $\kappa \ge c $ that 
$h^{1-\theta}|\partial A|_{\sC^\theta}\le C\kappa (\kappa +\bar{\mu})$. Then using arguments of subsection~\ref{sect-2-2} one can prove easily that for $\kappa \le h^{\sigma-\frac{1}{2}}$ 
\begin{equation*}
|F_{t\to h^{-1}\tau} \bar{\chi}_T(t) (hD_x)^\alpha (hD_x)^\beta 
\bigl(U(x,y,t)-U_{\varepsilon} (x,y,t)- U'_{\varepsilon} (x,y,t)\bigr)|\le C h^{1-d}
\end{equation*}
where we use the same $2$-term approximation, $T=\epsilon \bar{\mu}^{-1}$. Let  us take then $x=y$, multiply by $\varepsilon^{-d}\psi ( \varepsilon ^{1}(y-z))$ and integrate over $y$.  Using rough microlocal analysis one can prove easily that from both $U_{\varepsilon} (x,y,t)$ and $U'_{\varepsilon}$ we get $O(h^{1-d})$ and in the end of the day we arrive to the estimate 
$|\Delta A_\varepsilon |\le C\kappa\bar{\mu}$ which implies 
\begin{equation}
|\partial^2 A_\varepsilon|\le C\kappa \bar{\mu} |\log h| +C\mu 
\label{4-9}
\end{equation}
where obviously one can skip the last term. Here we used property of the Laplace equation. For our purpose it is much better than 
$|\partial^2 A_\varepsilon|\le C\kappa ^2 |\log h| +C\mu $ which one could derive easily.

Again using arguments of subsection~\ref{sect-2-2} one can prove easily that 
\begin{gather}
|\Tr(\psi H^-_{A,V}\psi )- \Tr(\psi H^-_{A_\varepsilon,V}\psi )|\le 
C\bar{\mu}^2 h^{2-d}\label{4-10}\\
\shortintertext{and therefore} 
|\Tr(\psi H^-_{A,V}\psi )- \int \Weyl_1(x)\psi^2(x)\,dx|\le  C\bar{\mu}^2 h^{2-d}\label{4-11}\\
\intertext{and finally for an optimizer}
\|\partial A\|^2 \le C\kappa \bar{\mu}^2 h .\label{4-12}
\end{gather}
Here $\mu$ and $\bar{\mu}$ were calculated for $A$, but it does not really matter as due to  $|\partial^2 A|\le C\kappa ^2h^{-\delta}$ we conclude that
$|\partial A- \partial A_\varepsilon |\le C\kappa ^2h^{-\delta}\varepsilon \le C$ due to restriction to $\kappa$.

Then, as $d=3$ 
\begin{equation}
\mu^2 \bigl(\mu /\kappa \bar{\mu}|\log h|\bigr)^3 \le \kappa \bar{\mu}^2 h
\label{4-13}
\end{equation}
and if $\mu \ge 1$ we have $\bar{\mu}=\mu$ and (\ref{4-13}) becomes 
$\kappa^{-3}|\log h|^{-3} \le C \kappa h$ which impossible under (\ref{4-7}). 

So, $\mu \le 1$ and (\ref{4-13}) implies (\ref{3-13}) and (\ref{4-12}), (\ref{4-13}) imply that for an optimizer 
$\|\partial A \|\le C(\kappa h)^{\frac{1}{2}}$ and 
$\mu \le C\kappa^4 h |\log h|^d$. So (i) is proven.

\medskip\noindent
(ii) Proof of (ii) follows then in virtue of arguments of subsection~\ref{sect-2-2}.

\medskip\noindent
(iii)  If $\kappa^*_h\le \kappa \le h^{-1}$ we apply partition-and-rescaling. So, $h\mapsto h'=h/\gamma$ and $\kappa \mapsto \kappa'= \kappa \gamma$ and to get into (\ref{4-7}) we need $\gamma = \epsilon \kappa^{-\frac{4}{3}}h^{-\frac{1}{3}}|\log (\kappa h)|^{-1}$ leading to the remainder estimate  $Ch^{-1}\gamma^{-2}$ which proves (ii).
\end{proof}

\begin{remark}\label{rem-4-3}
In versions 1 and 2 (v1 and v2) we lost a logarithmic factor in (\ref{4-7}) and (\ref{4-8}) (it was $h^{-\frac{1}{4}}$ and $h^{-\frac{1}{3}}\kappa ^{\frac{8}{3}}$).
\end{remark}

\bibliographystyle{alpha}



\end{document}